\documentclass[preprint,12pt]{elsarticle}

\usepackage[T2A]{fontenc}
\usepackage[utf8]{inputenc}
\usepackage[russian,english]{babel}

\usepackage{amsmath}
\usepackage{amssymb}
\usepackage{amsthm}
\usepackage{bm}
\usepackage{pict2e}
\usepackage{tikz}
\usepackage[e]{esvect}%

\usetikzlibrary{matrix}

\input cells.sty

\newtheorem{theorem}{Theorem}[section]
\newtheorem{corollary}{Corollary}[theorem]
\newtheorem{ex}[corollary]{Example}
\newtheorem{lemma}[theorem]{Lemma}

\theoremstyle{definition}
\newtheorem{remark}{Remark}[theorem]
\newtheorem*{definition}{Definition}
\newtheorem{examp}[corollary]{Example}

\def\NN{{\mathbb N}}
\let\geq\geqslant
\let\leq\leqslant
\def\hats{{\sc Hats\ }}

\def\HG{{\mathcal G}}

\def\gtimes#1{\mathop{\times\!_{_{#1}}}}
\def\hgn{\text{\rm HG}}
\def\cnst#1{\star #1}
\def\petal{{P\hskip-.6pt e}}

\newcommand{\nsfrac}[2]{\,\raise3pt\hbox{$\scriptstyle{#1}$\hskip-2pt}/\raise-3pt\hbox{\hskip-2pt$\scriptstyle{#2}$}}

\def\pathgame(#1,#2,#3,#4){\lower8pt\hbox{\setlength{\unitlength}{.7mm}\footnotesize
\begin{picture}(14,16)(-2,-6)
\put(0,0){\circle*{2}}\put(10,0){\circle*{2}}
\polyline(0,0)(10,0)
\put(-2,4){#1}\put(9,4){#2}\put(-2,-6){#3}\put(8,-6){#4}
\end{picture}}}

\makeatletter
\def\ps@pprintTitle{%
 \let\@oddhead\@empty
 \let\@evenhead\@empty
 \def\@oddfoot{}%
 \let\@evenfoot\@oddfoot}
\makeatother

\begin{document}

\begin{frontmatter}

\title{Hat guessing number of planar graphs is at least 22.
}
\author[inst1]{Aleksei Latyshev}
\ead{aleksei.s.latyshev@gmail.com}
\address[inst1]{ITMO University, St. Petersburg, Russia}

\author[inst3]{Konstantin Kokhas}
\ead{kpk@arbital.ru}
\address[inst3]{St. Petersburg State University, St. Petersburg, Russia}

\begin{abstract}
    We analyze the following version of the deterministic \hats game. We have a graph $G$, and a sage resides at each vertex of $G$. When the game starts, an adversary puts on the head of each sage a hat of a color arbitrarily chosen from a set of $k$ possible colors. Each sage can see the hat colors of his neighbors but not his own hat color. All of sages are asked to guess their own hat colors simultaneously, according to a predetermined guessing strategy and the hat colors they see, where no communication between them is allowed. The strategy is winning if it guarantees at least one correct individual guess for every color assignment. 
    
    In a modiﬁed version of the hat guessing game each sage makes $s\geq 1$ guesses. Given a graph $G$ and integer $s\geq 1$, the hat guessing number $\hgn_s(G)$ is the maximal number $k$ such that there exists a winning strategy.
    
    In this paper, we present new constructors, i.e. theorems that allow built winning strategies for the sages on different graphs.
    Using this technique we calculate the hat guessing number $\hgn_s(G)$ for paths and ``petunias'', and present a planar graph $G$ for which $\hgn_1(G) \ge 22$.
\end{abstract}
\begin{keyword}
graphs \sep deterministic strategy \sep hat guessing game \sep hat guessing number
\end{keyword}

\end{frontmatter}

\section{Introduction}

The \hats game goes back to an old popular Olympiad problem. Its generalization to arbitrary graphs attracted the interests of mathematicians recently (see, e.g.
\cite{bosek19_hat_chrom_number_graph, he20_hat_guess_books_windm,   alon2020hat}).

In this paper, we consider the following general version of \hats game.

Let $G = \langle {V, E} \rangle$ be a visibility graph, each vertex of $G$ is occupied by one sage, and only the sages in the adjacent vertices can see each other. We identify the sages with the graph vertices.

Let $h\colon V\to \NN$ be a \emph{``hatness'' function} that denotes the number of different hat colors a sage can get. For sage $A\in V$, we call the number $h(A)$ the \emph{hatness} of sage $A$. We may assume that the hat color of sage $A$ is a number from the set $[h(A)]=\{0, 1, 2, \dots, h(A)-1\}$.

Let $g\colon V\to \NN$ be a \emph{``guessing'' function} that determines the number of guesses each sage is allowed to make.

A function that is equal to a constant $m$ we denote by $\cnst{m}$.

\begin{definition}
  A hat guessing game or \hats for short is a triple $\HG=\langle {G, h, g} \rangle$, where $G$ is a visibility graph, $h$ is a hatness function, and $g$ is a guessing function. So, the sages are located in the vertices of the visibility graph $G$ and participate in \emph{test}. There is an adversary who plays against the sages. When test starts, the adversary puts on the head of each sage $v$ a hat of a color that he chooses from a set of $h(v)$ possible colors. Each sage can see the hat colors of his neighbors but not his own hat color.   
 The sages do not communicate, and each sage $s$ tries to guess the colors of their own hats by writing a list of $g(s)$ colors. If at least one of their guesses is correct, the sages \emph{win}, and the game is \emph{winning}. Before the test begins, the sages communicate in order to devise a guessing strategy, and this strategy is known to the adversary. It is assumed that the sages' strategy is always deterministic, so that the guess of the sage at a vertex $v$ is uniquely determined by the colors of the hats at the neighbors of $v$.
The games in which the sages have no winning strategy we call \emph{losing}. Color assignments for which none of the sages guess is a \emph{disproving hat placement}.

 We often denote the winning strategy of the sages in the game $\HG$ by the same symbol~$\HG$. 
The game $\langle {G, h, \star1 } \rangle$ is denoted $\langle {G, h} \rangle$. The classic \hats game is $\langle {G, \cnst{m}} \rangle$, the hatness function has constant value $m$ here.
\end{definition}

Given a graph $G$, its \emph{hat guessing number} $\hgn_s(G)$ is the maximal number $k$ such that there exists a winning strategy in the game $\langle {G, \cnst{k}, \cnst{s}} \rangle$ . Let $\hgn(G)=\hgn_1(G)$. Though hat guessing numbers $\hgn_s(G)$ are defined in terms of constant numbers of colors and guesses, the computation of them for an arbitrary graph $G$ is a hard problem and the general version of \hats game gives a more flexible approach for this. The exact values of $\hgn(G)$ are known for a few classes of graphs: for complete graphs, trees (folklore), cycles~\cite{cycle_hats}, and pseudotrees~\cite{Kokhas2018}. Also, there are some results for ``books'' and ``windmills'' graphs, see~\cite{he20_hat_guess_books_windm}, \cite{kokhas_hats_2021}. For bipartite, multipartite and $d$-degenerate graphs some estimations of hat guessing numbers are obtained by Alon et al.~\cite{alon2020hat}, Gadouleau and Georgiou \cite{Gadouleau2015}, He and Li \cite{he_Li_20_hat_guess_dgnr_graphs}, Knierim et al. \cite{Knierim_Martinsson_Steiner2021}, Bradshaw~\cite{Bradshaw_arxiv.2109.13422}.
M. Farnik \cite{Farnik2015} considered relation of the hat guessing number and the maximal degree of a graph and proved that $\hgn(G)< e \Delta(G)$ for any graph~$G$.  In \cite{LATYSHEV2022112868} we presented examples of graphs for which $\hgn(G)=(4/3)\Delta$, and proved that $\hgn(G)$ and diameter of $G$ are independent.

In our previous works~\cite{kokhas_cliques_I_2021}
and~\cite{kokhas_cliques_II_2021} (joint with V.\,Retinsky), we proved several theorems for \hats game $\langle {G, h} \rangle$, that allow to build new winning graphs by combining the graphs, for which their winning property is already proved. We call these theorems \emph{constructors}, they provide a powerful machinery for building strategies. The general \hats game $\HG=\langle {G, h, g} \rangle$ was suggested for the first time by Bla\v{z}ej  et al. in \cite{blazej_bears_2021}, where  the technique of constructors was combined with the independence polynomials approach. In Section 2, we present new constructors for general \hats game. One of them,  theorem~\ref{thm:CancelConstructor} generalizes the construction of clique join in \cite{blazej_bears_2021} and allows to reduce fractions $g(v)/h(v)$ for some vertices that looks very attractive in the spirit of ``theory of fractional hat guessing numbers'' developed in \cite{blazej_bears_2021}.

Let $\mathcal{OP}$ be the class of all outerplanar graphs and $\mathcal{P}\hskip-.5pt lanar$ be the class of planar graphs. If $ \mathcal{C} $ is a class of graphs, we define 
$$
\hgn_s(\mathcal{C})=\sup\{\hgn_s(G): G\in \mathcal{C}\}
$$
and $\hgn(\mathcal{C})=\hgn_1(\mathcal{C})$. It remains an open problem (stated explicitly in \cite{bosek19_hat_chrom_number_graph,Bradshaw_planar}), whether or not $\hgn(\mathcal{P}\hskip-.5pt lanar) <+\infty$. Approaching this problem, Bradshaw \cite{Bradshaw_planar} showed that $\hgn(\mathcal{OP}) <+\infty$. Knierim et al.
\cite{Knierim_Martinsson_Steiner2021} considered \hats game $\HG=\langle {G, \cnst{q}, g} \rangle$ and showed that $\hgn(\mathcal{OP}) <40$. The best known lower bound for the hat guessing number of planar graphs was 14, as proved in~\cite{kokhas_hats_2021}. In Section 3, we prove a better bound $\hgn(\mathcal{OP}) \geq 22$.

Following Bradshaw \cite{Bradshaw_planar} we deﬁne a \emph{petal} graph $\petal_n$ to be a graph obtained from path $P_n$ (containing $n\geq 1$ vertices) by adding a vertex $v$ adjacent to every vertex of $P_n$ (fig.~\ref{fig:petal-petunia}). We say that $v$ is the \emph{stem} of the petal. Then, we deﬁne a \emph{petunia} to be a graph in which every block is a petal graph. Thus, a petunia can be constructed by repeating the procedure that starts from a petal and subsequently adds new petals by gluing one of the vertices of a new petal (not necessarily the stem) to an arbitrary vertex of the petunia. A \emph{royal petunia} is a petunia that can be constructed by repeating a similar procedure that starts from a petal and subsequently adds new petals by gluing the stem of the new petal to an arbitrary vertex of the petunia, fig.~\ref{fig:petal-petunia}.

\begin{figure}[h]
\setlength{\unitlength}{.7mm}\footnotesize
\begin{center}
\begin{picture}(120,30)%
\multiput(0,0)(10,0){5}{\circle*{2}}
\put(20,20){\circle*{2}}
\polyline(10,0)(20,20)(0,0)(40,0)(20,20)(30,0)
\polyline(20,20)(20,0)
\multiput(80,20)(10,0){5}{\circle*{2}}
\put(100,30){\circle*{2}}
\polyline(90,20)(100,30)(80,20)(120,20)(100,30)(110,20)
\polyline(100,30)(100,20)
\multiput(70,10)(8,0){3}{\circle*{2}}
\polyline(80,20)(70,10)(86,10)(80,20)(78,10)
\multiput(80,0)(7,0){4}{\circle*{2}}
\polyline(87,0)(110,20)(80,0)(101,0)(110,20)(94,0)
\multiput(110,5)(7,0){5}{\circle*{2}}
\polyline(117,5)(110,20)(110,5)(138,5)(110,20)(131,5)
\polyline(110,20)(124,5)
\end{picture}
\end{center}

  \caption{Petal and royal petunia}
  \label{fig:petal-petunia}
\end{figure}
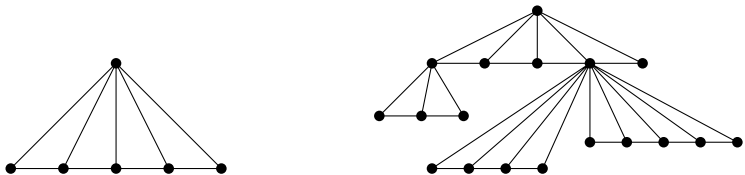

Let $\mathcal{P}\hskip-.5pt ath$, $\mathcal{P}\hskip-.5pt etal$, $\mathcal{RP}$ be the classes of all paths, petals and royal petunias correspondingly. For these classes of graphs we calculate in Section 3 their hat guessing numbers: 
\begin{align*}
\hgn_s(\mathcal{P}\hskip-.5pt ath)&= 4s-2, \\
\hgn_s(\mathcal{P}\hskip-.5pt etal)&= 4s(s+1)-2, \\
\hgn_s(\mathcal{RP})&= 4s(s+1)-2.
\end{align*}

\section{Constructors}

One of the most important and beautiful generalizations of the initial \hats puzzle is the following theorem about complete graphs $K_n$. It was proved by Kokhas and Latyshev \cite[theorem 2.1]{kokhas_cliques_I_2021} for the case $g=\cnst{1}$, and by Blažej,  Dvořák and Opler \cite[theorem 5]{blazej_bears_2021} in general case. This theorem supply us initial ``bricks'' for numerous constructions.

\begin{theorem}\label{thm:clique-win}
The \hats game $\langle {K_n, h, g} \rangle$ is winning if and only if 
\begin{equation}
    \sum_{v\in V(K_n)}\frac{g(v)}{h(v)}\geq 1.
    \label{eq:clique-win}
  \end{equation}
\end{theorem}

In figures we use notations from \cite{blazej_bears_2021}: near vertex $v$ we write fraction $g(v)/h(v)$. For example the game \pathgame($\frac s{s+1}$,$\frac 1{s+1}$,$A$,$B$) on the path $P_2$ with vertices $A$ and $B$ is winning by the previous theorem.

\subsection{Product with reducing}

Let $G_1 = (V_1 , E_1 )$ and $G_2 = (V_2 , E_2 )$ be graphs, let $S\subseteq V_1$ be a set of vertices inducing a clique in $G_1$, and let $v\in V_2$ be an arbitrary vertex of $G_2$. The \emph{clique join of graphs $G_1$ and $G_2$ with respect to $S$ and $v$} is the graph $G = (V, E)$ such that $V = V_1 \cup V_2 \setminus \{v\}$; and $E$ contains all the edges of $E_1$, all the edges of $E_2$ that do not contain $v$, and an edge between
every $w\in S$ and every neighbor of $v$ in $G_2$, see fig.~\ref{fig:clique-join}.

\begin{figure}[h]
\begin{center}
\setlength{\unitlength}{275bp}%
  \begin{picture}(1,0.18956614)%
    \put(0,0){\includegraphics[width=\unitlength,page=1]{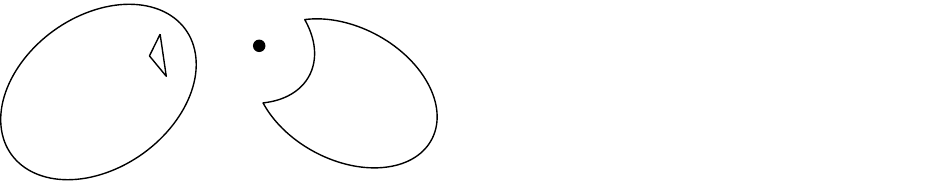}}%
    \put(0.07072003,0.08){$G_1$}%
    \put(0.36,0.08){$G_2$}%
    \put(0,0){\includegraphics[width=\unitlength,page=2]{clqjn.pdf}}%
    \put(0.09838629,0.13445178){$S$}%
    \put(0.26045467,0.17027046){$v$}%
    \put(0,0){\includegraphics[width=\unitlength,page=3]{clqjn.pdf}}%
    \put(0.77477991,0.0134315){$G$}%
    \put(0,0){\includegraphics[width=\unitlength,page=4]{clqjn.pdf}}%
    \put(0.66751359,0.11){$S$}%
    \put(0,0){\includegraphics[width=\unitlength,page=5]{clqjn.pdf}}%
  \end{picture}%
\end{center}
 \caption{The clique join of graphs $G_1$ and $G_2$ with respect to $S$ and $v$}
  \label{fig:clique-join}
\end{figure}

In \cite[Lemma 7]{blazej_bears_2021} Bla\v{z}ej et al. proved the following ``clique join'' constructor for two winning games.

\begin{lemma}[{\cite[Lemma 7]{blazej_bears_2021}}]
\label{lem:CliqueJoin}
Let $G =(V, E)$ and $G' = (V', E')$ be two graphs, $S \subseteq V$ be a set inducing a clique in $G$, $v\in V'$, and $\widetilde G$ to be the clique join of graphs $G$ and $G'$ with respect to $S$ and $v$. If the sages win the games $\bigl(G, h, g\bigr)$ and $\HG' = \bigl(G', h', g' \bigr)$, then they also win the game $\widetilde \HG = (\widetilde G, \widetilde h, \widetilde g)$ where
$$
\widetilde h(u), \widetilde g(u) =
\begin{cases}
  h(u), \ g(u)                &u\in V\setminus S,\\
  h'(u), \ g'(u)              &u\in V'\setminus \{v\},\\
  h(u) h'(v), \ g(u) g'(v)    &u \in S.
\end{cases}
$$
\end{lemma}

If $S$ consists of one vertex, the following corollary holds. We formulate it in terms of \cite[theorem 3.1]{kokhas_cliques_I_2021} where it was proved in case $g =g'=\cnst1$. 
For two games $\HG =\left\langle G, h, g \right\rangle$ and $\HG' =\left\langle G', h', g' \right\rangle$ where $V(G)\cap V(G')=\{v\}$ we say that game $\widetilde \HG =\HG\gtimes{v} \HG'$ is \emph{a product of games $\HG$ and $\HG'$ with respect to vertex~$v$} if $\widetilde \HG = \langle {G\cup G', \widetilde h, \widetilde g} \rangle$, where 
$$
\widetilde h(u), \widetilde g(u) =\begin{cases}
  h(u), \ g(u)          &u\in V(G)\setminus \{v\},\\
  h'(u), \ g'(u)        &u\in V(G')\setminus \{v\},\\
  h(v) h'(v), \ g(v) g'(v)   &u =v.
\end{cases}
$$

\begin{corollary}\label{thm:multiplication}
 Let $\HG = \bigl(G, h, g\bigr)$ and $\HG' = \bigl(G' h', g'\bigr)$ be two hat guessing games such that $V(G)\cap V(G')=\{v\}$. If the sages win the games $\HG$ and $\HG'$, then they also win the game $\widetilde \HG = \HG\gtimes{v} \HG'$.
\end{corollary}

Now we are going to generalize the lemma~\ref{lem:CliqueJoin}, by adding a possibility for some set of vertices $I$ to reduce simultaneously the values of $g(u)$ and $h(u)$, $u\in I$ by common factors. 

So, fix a strategy of sages in the game $\langle G, h, g\rangle$ and let $I\subset S\subset V(G)$. We call a pair $(S,I)$ \emph{predictable} if there exists an algorithm such that for any hat colors assignment on $V(G)$ 

1) this algorithm uses the information about the hats assignment on the set $S$ only and determines a sage $u$ that has guessed his color correctly provided that at least one sage from $S$ has guessed his color correctly;

2) in the case $u\in I$ the algorithm determines the set $A_u$ of $u$'s guesses according to his strategy.

The idea of lemma~\ref{lem:CliqueJoin} proof is to write hat colors for $u\in S$ as pairs $(c,c')$, where $c\in [g(u)]$, $c'\in [g'(u)]$, we call these colors \emph{composite}. The sages in~$S$ calculate the coordinates $c$, $c'$ of their composite colors using strategies $\HG$, $\HG'$ separately. The set $S$ in the following theorem  %
has no to be a clique, 
but it is predictable. The sages in $S$ have composite hat colors $(c,c')$, the sages in $N_{G'}(z)$ see the whole $S$, therefore they understand who in $S$ is potentially a winner in subgame $\HG$ and in fact they play subgame $\HG'$ with this winner according to the $G'$-coordinate of his color.

\begin{theorem}\label{thm:CancelConstructor}
Let $G=(V,E)$, $G'=(V',E')$ be two graphs, the games $\HG=\langle G, h, g\rangle$, $\HG'=\langle G', h', g'\rangle$ be winning, and $(S,I)$ be a predictable pair in $G$, $z\in V'$. Assume that for each vertex $v\in I$ the numbers $g(v)$ and $h'(z)$ have a common divisor $s_v$: 
$$
g(v)=s_v a_v, \qquad h'(z)=s_vb'_v
$$
where $a_v$, $b'_v$ are natural numbers. Let $\widetilde G$ be a graph such that 
$$
V(\widetilde G)=V\cup V'\setminus\{z\}, \quad
E(\widetilde G)=E\cup E(G'\setminus\{z\}) \cup \{uw \mid  u\in S, w\in N_{G'}(z) \}.
$$ 
Then the game $\widetilde{\HG}=\langle \widetilde G, \widetilde h,  \widetilde g\rangle$ is winning where
$$
\widetilde h (u), \widetilde g (u) =\begin{cases}
h(u), \ g(u)        &  u\in V\setminus S,\\
h'(u), \ g'(u)     &  u\in V'\setminus \{z\},\\
h(u) h'(z), \ g(u) g'(z)&  u\in S\setminus I,\\
h(u) b'_u, \ a_u g'(z) &  u\in I, 
\end{cases}
$$
\end{theorem}

\goodbreak

Thus, the hatnesses and the guessing numbers of vertices $u\in S\setminus I$ satisfy the product rule as in a usual clique join, but for vertices $v\in I$  these products are reduced by $s_v$: $\bigl(\widetilde h (v), \,\widetilde g(v)\bigr) =\bigl(h(v) b'_v, \,a_v g'(z)\bigr)$ instead of $\bigl(h(v) s_vb'_v, \,s_v a_v g'(z)\bigr)$.

\begin{proof}
We may think that hats of each sage $u\in S\setminus I$ in game $\widetilde{\HG}$ have composite colors from the set  $[h(u)]\times [h'(u)]$ and for sages $u\in I$ the hats have <<reduced>> composite color from $[h(u)]\times [b'_u]$.
We will describe the strategy  $\widetilde{\HG}$, it consists of two stages.

At the first stage, only the sages in $V$ act: they look at the first coordinates of the colors of their neighbors in $G$ and calculate their guesses according to strategy $\HG$. Sages from $V\setminus S$ will tell these guesses to adversary, and if at least one of them has guessed correctly, the game $\widetilde{\HG}$ is winning independently of the other's guesses. 

The sages from $S$ go to the second stage. Observe that if there are sages in $S$ who have guessed the first coordinates of their colors correctly, then by the definition of a predictable pair $(S,I)$ the ``spectators'' from $N_{G'}(z)$ can independently choose a sage $u\in S$ (the same for all spectators) who has guessed his color correctly.
This correct guess means that sage $u$ has constructed a set $A_u=\{c_1, \dots, c_{s_ua_u} \}$, where $c_i\in [h(u)]$ for all $i$, such that $G$-coordinate of $u$'s hat color belongs $A_u$. And in case $u\in I$ all sages in $N_{G'}(z)$ know this set $A_u$.

At the second stage, the sages from $S\cup V'\setminus \{z\}$ act. In fact they play on $G'$ by strategy $\HG'$. 

Each sage $u\in S\setminus I$ looks at hat colors on $N_{G'}(z)$ and calculates 
the set $B_u$ of $G'$-coordinates of his guesses by using strategy~$\HG'$ of sage~$z$. After that $u$ constructs a set $A_u\times B_u$ and declares it to be  the set of his guesses by strategy~$\widetilde{\HG}$.  It is clear that if $u$ guesses his colors correctly at both stages, then $A_u\times B_u$ contains the composite color of his hat and $u$ wins in the game $\widetilde{\HG}$.

Each sage $u\in I$ acts as follows. Sage $u$ reconstructs set $A_u$ in a predetermined way into a Cartesian product $A_u=[a_u]\times[s_u]$. Then, $u$ and those who see him, may interpret $u$'s hat color as element of $[a_u]\times[s_u]\times[b'_u]$. Since $s_u b'_u=h'(z)$, one may think that the latter set is a set $[a_u]\times Z$, where $Z$ is a set of $z$'s hat colors in game~$\HG'$. Thus we present $u$' hat color as a~composite color which second coordinate is good for playing on the graph $G'$ (in the place of~$z$) by strategy $\HG'$. Now $u$ looks at his neighbors in $N_{G'}(z)$ and calculates by strategy $\HG'$ of the sage $z$ the set $B_u$ of $g'(z)$ guesses. Then the set $A_u\times B_u$ is assumed to be his set of guesses by strategy ~$\widetilde{\HG}$. As in the previous case, if $u$ has guessed his colors correctly at both stages, then $A_u\times B_u$ contains the composite color of his hat and $u$ wins in the game $\widetilde{\HG}$.

Now we describe the strategy of sages from $V'\setminus \{z\}$. The sages from  ${V'\setminus N_{G'}(z)}$ just use strategy $\HG'$. The sages from $N_{G'}(z)$ (they a those who can see sage $z$ in graph $G'$) at the end of first stage choose sage $u\in S$. At the second stage, this sage $u$ plays the role of the sage $z$ and the sages in $N_{G'}(z)$ will only look at the component $S$ at the $u$'s hat. 
Moreover, they take into account either $G'$-coordinate of $u$'s hat color if $u\in S\setminus I$, or reconstructed $G'$-coordinate if $u\in I$. Thus, they act as in game $\HG'$. We have explained in previous paragraphs what happens if $u$ guesses correctly. And if some sage from $V(G')\setminus \{z\}$ guesses correctly, the sages win in game $\widetilde{\HG}$.
\end{proof}

For two games $\HG=\langle G, h, g\rangle$ and $\HG'=\langle G', h', g'\rangle$ one of possible ways to apply theorem \ref{thm:CancelConstructor} is to substitute the whole graph $G$ in place of some vertex $z\in V(G')$. In this case $S=I=V(G)$ is of course a predictable pair, and we obtain the following corollary.

\begin{corollary}\label{cor:to-construct-petunia}
Let $G=(V,E)$, $G'=(V',E')$ be two graphs, the games $\HG=\langle G, h, g\rangle$, $\HG'=\langle G', h', g'\rangle$ be winning, $z\in V'$. Assume that for each vertex $v\in V$ the numbers $g(v)$ and $h'(z)$ have common divisor $s$: 
$$
g(v)=sa_v, \qquad h'(z)=sb'_v
$$
where $a_v$, $b'_v$ are natural numbers. Let $\widetilde \HG=\langle \widetilde G,\widetilde  h, \widetilde g\rangle$ be a game obtained by substituting graph $G$ in place of vertex $z$, and the functions $\widetilde h$ and $ \widetilde g$ are defined by 
$$
\widetilde h (u), \widetilde g (u)=\begin{cases}
h'(u), \ g'(u)     &  u\in V'\setminus \{z\},\\
h(u) b'_u, \ a_u g'(z) &  u\in V. 
\end{cases}
$$
Then game  $\widetilde{\HG}$ is winning.
\end{corollary}

Simplified version of this substitution constructor (without reducing) was proved in \cite[Theorem 3.2]{kokhas_cliques_I_2021}.

\begin{examp}\label{ex:1-22-planar}
Let $\HG$ be a game on petal $G=\petal_{n}$, where all sages have hatness~22 and 2~guesses (see fig.~\ref{fig:petal-plus}). By theorem \ref{thm:s-petal} below this game is winning for large $n$. Let $\HG'$ be a game \pathgame($\frac 1{2}$,$\frac 1{2}$,,), it is winning by theorem~\ref{thm:clique-win}. The substitution with reducing gives us a winning game on graph $\widetilde{G}$ (fig.~\ref{fig:petal-plus}).
\end{examp}

\begin{figure}
\setlength{\unitlength}{.7mm}\footnotesize
\begin{center}
\begin{picture}(150,60)(0,-35)
\multiput(0,0)(10,0){2}{\circle*{2}\put(-4,-8){$\nsfrac{2}{22}$}}
\multiput(30,0)(10,0){2}{\circle*{2}\put(-4,-8){$\nsfrac{2}{22}$}}
\put(20,20){\circle*{2}}
\polyline(10,0)(20,20)(0,0)(40,0)(20,20)(30,0)
\multiput(60,0)(0,-30){2}{\circle*{2}}
\polyline(60,0)(60,-30)
\put(17,-7){$\ldots$}
\put(22,20){$\nsfrac{2}{22}$}
\put(61,-1){$\nsfrac{1}{2}$}
\put(63,-27){$\nsfrac{1}{2}$}
\put(55,-3){$z$}
\put(10,-40){$G=\petal_{n}$}\put(52,-40){$G'=P_2$}\put(128,-40){$\widetilde G$}
\put(131,21){$\nsfrac{1}{22}$}
\put(131,-33){$\nsfrac{1}{2}$}
\put(127,-5){$\ldots$}
\put(110,0){
\multiput(0,0)(10,0){2}{\circle*{2}\put(-6,-6){$\nsfrac{1}{22}$}}
\multiput(30,0)(10,0){2}{\circle*{2}\put(-1.5,-6){$\nsfrac{1}{22}$}}
\put(20,20){\circle*{2}}
\put(20,-30){\circle*{2}}
\polyline(0,0)(20,-30)(10,0)(20,20)(0,0)(40,0)(20,20)(30,0)(20,-30)(40,0)
}
\put(130,-5){\oval(60,50)[l]}
\end{picture}
\end{center}
  \caption{Substitution with reducing of graph $G$ on place of vertex $z$}
  \label{fig:petal-plus}
\end{figure}
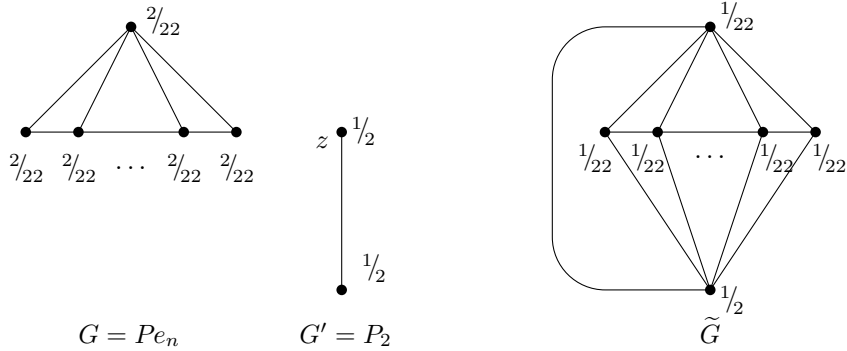

\subsection{Half-edge removal}

In this section we assume that a hat guessing game is played on a directed graph. We denote directed edge $u\to v$ by $\vv{uv}$, it means that vertex $u$ sees vertex $v$. For $E\subset V(G)$ we denote by $G[E]$ the induced subgraph of $G$ with vertex set $E$. If $\HG=\langle {G, h, g} \rangle$, then $\HG[E]=\langle {G[E], h\vert_E, g\vert_E} \rangle$.

Let us notice that a hat guessing game ``in fact'' always takes place on a directed graph, because for each vertex the set of visible vertices is determined by the visibility graph and the property of visibility is not necessarily symmetric.

We start with simple observation.

\begin{lemma}\label{lem:scc}
  Let $\HG = \langle {G, h, g} \rangle$ be a game, and $\langle {S, T} \rangle$
  be a cut of $G$ such that for each vertices $s\in S$ and $t\in T$ directed edge $\vv{ts} \notin E(G)$. Then $\HG$ is winning if and only if at least one of the games $\HG[S]$ or $\HG[T]$ is winning.
\end{lemma}
\begin{proof}
  Obviously, if one of games $\HG[S]$ or $\HG[T]$ is winning, then $\HG$ is
  winning too. %

  Now we prove that if both games $\HG[S]$ and $\HG[T]$ are losing, then $\HG$ is losing too. For an arbitrary strategy $f$ for the game $\HG$ we will construct a disproving hats placement. 
    Since no sage in $T$ sees anybody in $S$ we may consider $f\vert_T$ as a strategy in the game $\HG[T]$.
  By assumption, this strategy is losing, so there exists a disproving hat placement $c$ on $T$. If we fix the hat placement $c$ on~$T$ in game $\HG$ then for all possible hat placements on $S$ strategy $f$ determines the guesses of all sages in $S$. In fact we obtain a strategy in $\HG[S]$ that is also losing, 
  so there exists a disproving hat placement~$\tilde c$. Then $c\cup \tilde c$ is a~disproving hat placement for $f$.
\end{proof}

\begin{corollary}
  The game $\HG$ on graph $G$ is winning if and only if for each strongly connected component $H$ of graph $G$ the game $\HG[H]$ is winning. 
\end{corollary}

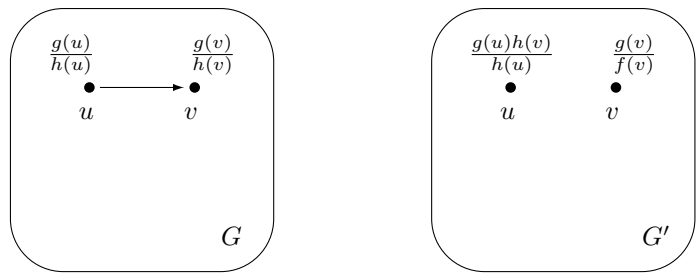
\begin{figure}
\centering
\setlength{\unitlength}{.7mm}\footnotesize
\begin{picture}(120,50)(-10,-35)
\put(0,0){\circle*{2}}\put(20,0){\circle*{2}}
\put(2,0){\vector(1,0){16}}
\put(-8,5){$\frac{g(u)}{h(u)}$}\put(19,5){$\frac{g(v)}{h(v)}$}\put(-2,-6){$u$}\put(18,-6){$v$}
\put(10,-10){\oval(50,50)}
\put(25,-30){$G$}

\put(80,0){
\put(0,0){\circle*{2}}\put(20,0){\circle*{2}}
\put(-8,5){$\frac{g(u)h(v)}{h(u)}$}\put(19,5){$\frac{g(v)}{f(v)}$}\put(-2,-6){$u$}\put(18,-6){$v$}
\put(10,-10){\oval(50,50)}
\put(25,-30){$G'$}}
\end{picture}
    \caption{Half edge removal}
    \label{fig:half-edge-rem}
\end{figure}

The following theorem allows to split a strongly connected graph
(e.g. connected undirected graph) onto strongly connected components.
Though the theorem is almost trivial, it is a real gem that provides short elegant proofs of many known results.

Given a game $\HG = \langle {G, h, g} \rangle$ such that $u$, $v\in V(G)$ and  $\vv{uv}\in E(G)$, we consider new game  $\HG' = \langle G', h, g' \rangle$,  where graph $G'$ is obtained from $G$ by removing edge $\vv{uv}$ (see fig.~\ref{fig:half-edge-rem}), function $h$ has not been changed and $g'$ is obtained from $g$ by a minor transformation:
\begin{equation}\label{eqn:guess-edge-rem}
 g' (x)=\begin{cases}
g(x)     &  x\in V(G)\setminus \{u\},\\
g(u) h(v) &  x=u.  
\end{cases}
\end{equation}
If the edge $\vv{vu}$ was present in graph $g$, this operation does not affect it. 

\begin{theorem}[on half-edge removal]\label{thm:half-edge-removal}
  Let $\HG = \langle {G, h, g} \rangle$ be a winning game, and $\vv{uv}$ be an edge of $G$. Let $\HG' = \langle {G \setminus \vv{uv}, h, g'}\rangle$ be a game obtained by removing directed edge $\vv{uv}$ from graph $G$, where guess function $g'$ is given by~\eqref{eqn:guess-edge-rem}. Then game $\HG'$  is winning.
\end{theorem}
\begin{proof}
  Fix a winning strategy for the game $\HG$. In the game $\HG'$ the sages will use the same strategy everywhere behind the vertex $u$. The vertex $u$ in the game $\HG'$ cannot see $v$, so let the sage $u$ just say the union of guesses for all possible colors of~$v$. Obviously, the number of his guesses is at most $g(u) h(v) = g'(u)$, and this strategy is winning.
\end{proof}

\begin{corollary}\label{cor:half-edge-removing}
  Let $\HG = \langle {G, h, g} \rangle$ be a losing game and $\vv{uv} \notin E(G)$. Let %
  $\HG' = \langle {G \cup \{\vv{uv}\}, h, g'\} }\rangle$, where $g'(x)=g(x)$ for all vertices $x\ne u$ and $g'(u)=\lfloor g(u) / h(v) \rfloor$. Then $\HG'$ is losing.
\end{corollary}

As examples of application of theorem~\ref{thm:half-edge-removal} we reprove a couple of known statements.

\begin{ex}[{losing part of~\cite[theorem 1.1]{alonHatGuessingNumber2022}}]
  Let $G$ be a graph depicted in fig.~\ref{fig:alon-planar-graph} (with an arbitrary number of horizontal edges). Then the game $\langle {G, \cnst{13}} \rangle$ is losing.
  \label{ex:alon}
\end{ex}

\begin{figure}[h]
\setlength{\unitlength}{.7mm}\footnotesize
\begin{center}
\begin{picture}(70,40)(0,-20)
\multiput(10,0)(25,0){2}{\circle*{2}\put(10,0){\circle*{2}}\polyline(0,0)(10,0)\put(9,-5){12}}
\put(70,0){\circle*{2}\put(10,0){\circle*{2}}\polyline(0,0)(10,0)\put(9,-5){12}}
\put(0,20){\circle*{2}}
\put(0,-20){\circle*{2}}
\polyline(0,-20)(10,0)(0,20)(0,-20)(20,0)(0,20)
\polyline(0,-20)(35,0)(0,20)(0,-20)(45,0)(0,20)
\polyline(0,-20)(80,0)(0,20)(0,-20)(70,0)(0,20)
\put(-6,18){$A$}\put(2,22){$2$}
\put(-6,-23){$B$}\put(2,-24){$3$}
\put(3,-2){13}\put(26,-1){13}\put(65,-7){13}
\put(52,0){\dots}
\end{picture}
\end{center}

  \caption{In example \ref{ex:alon} we consider this game instead of $\langle {G, \cnst{13}} \rangle$}
  \label{fig:alon-planar-graph}
\end{figure}
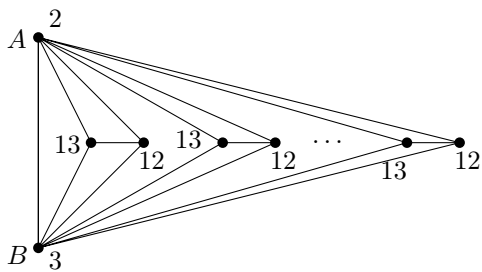

\begin{proof}
We will prove the stronger fact that the game $\langle {G, h} \rangle$ is losing, where hatness function $h$ is depicted in fig.~\ref{fig:alon-planar-graph}. After removing by theorem~\ref{thm:half-edge-removal} all the incoming directed edges to $A$ and to $B$ we obtain the game in which each strongly connected component is either edge \pathgame($\frac 1{2}$,$\frac 1{3}$,,) or edge \pathgame($\frac 6{12}$,$\frac 6{13}$,,). Both games are losing by theorem~\ref{thm:clique-win}.
\end{proof}

\begin{ex}[{\cite[theorem 6]{bosek19_hat_chrom_number_graph}}]
  Let $G$ be a connected graph, and let $V = A \cup B$ be a partition of the
  vertex set of $G$. Let $d = \max_{v\in B}\{|N(v) \cap A|\}$. Then $\hgn_s(G)
  \leq \hgn_{s_1}(G[B])$, where $s_1 = s(\hgn_s(G[A]) + 1)^d$.
\end{ex}
\begin{proof}
  Let $a = HG_s(G[A]) + 1$,  $b = HG_{s \cdot a^d}(G[B]) + 1$ and $h\colon V\to \mathbb N$ be a hatness function such that $h\vert_A=a$, $h\vert_B=b$.
  It is sufficient to prove that the game $\HG = \langle {G, h, \cnst{s}} \rangle$ is losing. Let us remove (by theorem~\ref{thm:half-edge-removal}) all directed edges from $B$ to $A$. Then each strongly connected component in the obtained graph is either a subset of $A$ or a subset of $B$, and the value of the guessing function in each vertex of $B$ is increased by at most $a^d$ times. The obtained game is losing since the games $\langle {G[A], \cnst{a}, \cnst{s}} \rangle$ and $\langle
  {G[B], \cnst{b}, \cnst{(s \cdot a^d)}} \rangle$ are losing. Then the game $\HG$ is losing by several applications of corollary~\ref{cor:half-edge-removing}.
\end{proof}

\subsection{Gluing of losing games}

We say that a vertex $A$ with a fixed guessing number $s$ is \emph{strong} if its hatness is $s+1$. The next constructor demonstrates that two losing games can be glued by the vertex $A$, if in both games the guessing number of $A$ is $s$, and in at least one of the games $A$ is a strong vertex. For $s=1$ this statement was proven in \cite[Theorem 4.1]{kokhas_cliques_II_2021}.

\begin{theorem}\label{thm:s-lose-sum}
Let $G=  G_1+_AG_2$, where $G_1$ and $G_2$ are graphs, for which $V(G_1)\cap V(G_2)=\{A\}$. Let games $\HG_1=\langle G_1, h_1, g_1\rangle$ and $\HG_2=\langle G_2, h_2,g_2\rangle $ be losing, and let the following conditions hold:
$$
g_1(A)=g_2(A)=s, \qquad  h_1(A)\geq h_2(A)=s+1.
$$ 
Then game $\HG=\langle G_1+_AG_2, h,g\rangle$ is losing, where
$$
  h(x) =
  \begin{cases}
    h_i(x),&x\in V_i\setminus\{A\}  \ (i=1, 2), \\
    h_1(A),&x= A,
  \end{cases}
  \quad  g(x) =g_i(x), \  x\in V_i  \ (i=1, 2).
$$
\end{theorem}

\begin{proof}
Assuming the contrary let $f$ be a winning strategy in game $\HG$. Denote by $N_1$ the set of neighbours of vertex $A$ in graph $G_1$. For any hats placement $\varphi$ on $V(G_1)$ the guesses of all the sages in $V(G_1)\setminus\{ A\}$ are determined by strategy $f$. We will show that there exist $s+1$ hats placements $\varphi_i$ $(i=1,\dots,s+1)$ on~$V(G_1)$, such that for $i\ne j$
$$
\varphi_i\big|_{N_1}=\varphi_j\big|_{N_1}, \qquad  \varphi_i(A)\neq\varphi_j(A), 
$$
and such that if the sages from $G_1$ play according strategy $f$, then for all these placements none of the sages from  $V(G_1)\setminus\{ A\}$ guesses correctly.

For each hats placement $\alpha$ on vertices of $N_1$ denote by $C(\alpha)$ the set of hat colors of sage $A$, such that for all placements $\beta$ on $G_1$, for which
$$
\beta \big|_{N_1} = \alpha, \qquad \beta(A)\in C(\alpha),
$$
none of the sages from set $V(G_1)\setminus\{ A\}$ guesses correctly by strategy $f$. Suppose that the statement from the previous paragraph does not hold. Then each set $C(\alpha)$ contains at most $s$ colors. Consider the following strategy for game $\HG_1$: let all the sages from~$G_1$, except $A$, play by strategy~$f$, and sage $A$ names the colors from set $C(\alpha)$ (supplementing them with arbitrary colors, if $C(\alpha)$ contains less than $s$ elements). This strategy is winning, because if nobody in $V(G_1)\setminus \{A\}$ has guessed correctly, then $A$'s color belongs to $C(\alpha)$, and he guesses correctly. Contradiction.

Consider these $s+1$ placements $\varphi_i$. Fix a hats placement $\alpha=\varphi_i\big|_{N_1}$ on $N_1$ and restrict ourselves to only those hats placements on $G_2$, where sage $A$ receives a hat of one of $s+1$ colors $\varphi_i(A)$,  $i=1$, \dots, $s+1$. Then the strategy $f$ defines the actions of the sages on the graph $G_2$, i.\,e.\ in the losing game $\HG_2$ subject with the only restriction that in the case $h_1(A)> s+1$ sage $A$ by this strategy name the colors from the set that contains more than $s+1$ colors, i.\,e.\ more than his hatness in game $G_2$. But the extra colors do not help to win. Therefore there exists disproving placement~$\psi$ on $G_2$. If $\psi(A)=\varphi_j(A)$, then $\psi\cup \varphi_j\big|_{V(G_1)\setminus A}$ is a disproving hats placement for strategy $f$ in game~$\HG$.
\end{proof}

We need one more lemma that demonstrates a specific property of strong vertices.

\begin{lemma}\label{lem:(s+1 s)vertex-removing}
  Let $\HG = \langle {G, h, g} \rangle$ be a game, $G' = G \setminus A$, where $A \in V(G)$ is a vertex connected with all other vertices of $G$, $h(A)=s+1$, $g(A)=s$, and  $\HG' =  \langle G', h', (s+1)\cdot g'\rangle$, where    $h'=h\Big\vert_{V(G')}$, $g'=g\Big\vert_{V(G')}$. Then the games $\HG$ and $\HG'$ are winning (or losing) simultaneously.
\end{lemma}

\begin{proof}
  If $\HG$ is a winning game, we remove by theorem \ref{thm:half-edge-removal} all half-edges $\vv{vA}$ and obtain the winning game. In this game $A$ has $s$ guesses, $s+1$ colors, and no information, therefore we can assign $A$'s color so that $A$ does not guess. But now the remaining sages play the game $\HG'$. Hence $\HG'$ is winning.
  
  If the game $\HG'$ is winning, substitute $\HG'$  in the winning game   \pathgame($\frac s{s+1}$,$\frac 1{s+1}$,$A$,$B$) in place of vertex $B$ by corollary \ref{cor:to-construct-petunia}. We obtain a winning game $\HG$.
  \end{proof}

\subsection{Transfer of a hint along the new edge}

We will prove here one technical tool that is useful in analysis games on paths.

Let game $\HG=\langle G, h, g\rangle$ be winning under condition that the adversary makes the following hint during the game. For one vertex $B\in V(G)$ a natural number $w_B\leq h(B)$ is fixed and it is known that the adversary will come to the sage $B$ during the game and will tell him a set of $w_B$ consecutive remainders (i.\,e.~the set of remainders in the form $x$, $x+1$, \dots, ${x+{w_B}-1} \bmod h(B)$), containing the color of his hat; the other sages will not hear this hint. Vertex~$B$, number $w_B$ and the rule of proclaiming of the hint are known to the sages beforehand. Denote a game with hint by $\langle G, h, g, B, w_B\rangle$.

For example, the game $\langle G, h, g, B, w_B\rangle$ is certainly winning in the case $w_B\leq g(B)$.

\begin{theorem}\label{thm:hint_constructor}
Let graph $G$ contain vertex $B$, and graph $\tilde G$ be obtained from graph $G$ by appending new vertex $A$ and new edge $AB$. Let hatness function $\tilde h$ and guessing function~$\tilde g$ be given on graph $\tilde G$, and let $h=\tilde h\Big\vert_{V(G)}$, $g=\tilde g\Big\vert_{V(G)}$. Let for some natural numbers $w_A$, $w_B$ such that $g(A)\leq w_A\leq h(A)$ and $g(B)\leq w_B\leq h(B)$, the conditions hold:

\smallskip
{\rm (i)} game with hint $\HG=\langle G, h, g, B, w_B\rangle$ is winning,

\smallskip
{\rm (ii)} $w_B\cdot h(A)$ is divisible by $h(B)$,

\smallskip
{\rm (iii)} $w_A w_B\geq (w_A-g(A))h(B)$.

\smallskip\noindent
Then the game with hint $\tilde\HG=\langle \tilde G, \tilde h, \tilde g, A, w_A\rangle$ is winning.
\end{theorem}

\emph{Proof}. To describe the strategy of sage $A$, construct table $h(A)\times h(B)$, in which some squares are empty, and the others contain letters ``$L$'' by the following rule. Number the rows of the table by numbers from $0$ to $h(A)-1$, we identify the numbers of rows  with possible colors of $A$'s hat. Number the columns of the table by numbers from $0$ to $h(B)-1$, i.e., by possible colors of $B$'s hat. For each $i$ ($0\leq i\leq h(A)-1$) we put letters ``$L$'' in the cells of $i$-th row in columns with numbers
\begin{equation}\label{eqn:table-AB}
i w_B, \quad  iw_B+1, \quad \ldots, \quad iw_B+w_B-1 \pmod{h(B)}
\end{equation}
(i.\,e. $w_B$ letters ``$L$'' in total), see\ fig.~\ref{fig:interaction-tableAB}. One may consider the obtained table as torus: calculations modulo $h(B)$ in rule \eqref{eqn:table-AB} allow one to identify $h(B)$-th column with zeroth column, and condition (ii) allows one to identify $h(A)$-th row with zeroth row.

\begin{figure}[h]%
\begin{center}
\quad\quad The colors of sage $B$

\vskip7mm

\quad
\raise50pt\vbox{\hsize =2cm\noindent The\\ colors\\ of\\ sage\\ $A$}
\qquad
\vbox{\cellsize=12pt\footnotesize
\cput(0,1){0}\cput(0,2){1}\cput(0,3){2}\cput(0,5){$w_B$}\cput(0,9){$\dots$}\cput(0,14.5){$h(B) - 1$}
\cput(1,0){0}\cput(2,0){1}\cput(3,0){2}\cput(4.5,-0.5){$\dots$}\cput(6,-.3){$w_A$} \cput(10,-0.5){$\dots$} \cput(14,-1.3){$h(A) - 1$}
\cput(1,1){$L$}\cput(1,2){$L$}\cput(1,3){$L$}\cput(1,4){$L$}\cput(1,5){$L$}
\cput(2,7){$L$}\cput(2,8){$L$}\cput(2,9){$L$}\cput(2,10){$L$}\cput(2,6){$L$}
\cput(3,1){$L$}\cput(3,11){$L$}\cput(3,12){$L$}\cput(3,13){$L$}\cput(3,14){$L$}
\cput(4,2){$L$}\cput(4,3){$L$}\cput(4,4){$L$}\cput(4,5){$L$}\cput(4,6){$L$}
\cput(5,7){$L$}\cput(5,8){$L$}\cput(5,9){$L$}\cput(5,10){$L$}\cput(5,11){$L$}
\cput(6,12){$L$}\cput(6,13){$L$}\cput(6,14){$L$}\cput(6,1){$L$}\cput(6,2){$L$}
\cput(7,3){$L$}\cput(7,4){$L$}\cput(7,5){$L$}\cput(7,6){$L$}\cput(7,7){$L$}
\cput(8,8){$L$}\cput(8,9){$L$}\cput(8,10){$L$}\cput(8,11){$L$}\cput(8,12){$L$}
\cput(9,13){$L$}\cput(9,14){$L$}\cput(9,1){$L$}\cput(9,2){$L$}\cput(9,3){$L$}
\cput(10,4){$L$}\cput(10,5){$L$}\cput(10,6){$L$}\cput(10,7){$L$}\cput(10,8){$L$}
\cput(11,9){$L$}\cput(11,10){$L$}\cput(11,11){$L$}\cput(11,12){$L$}\cput(11,13){$L$}
\cput(12,14){$L$}\cput(12,1){$L$}\cput(12,2){$L$}\cput(12,3){$L$}\cput(12,4){$L$}
\cput(13,5){$L$}\cput(13,6){$L$}\cput(13,7){$L$}\cput(13,8){$L$}\cput(13,9){$L$}
\cput(14,10){$L$}\cput(14,11){$L$}\cput(14,12){$L$}\cput(14,13){$L$}\cput(14,14){$L$}
\cells{
 _ _ _ _ _ _ _ _ _ _ _ _ _ _
|_:_:_:_:_|_:_:_:_:_:.:.:.:.|
|_:.:.:.:.|_:_:_:_:_|_:_:_:_|
|_|_:_:_:_:_:.:.:.:.|_:_:_:_|
|.|_:_:_:_:_|_:_:_:_:_:.:.:.|
|_:_:.:.:.:.|_:_:_:_:_|_:_:_|
|_:_|_:_:_:_:_:.:.:.:.|_:_:_|
|.:.|_:_:_:_:_|_:_:_:_:_:.:.|
|_:_:_:.:.:.:.|_:_:_:_:_|_:_|
|_:_:_|_:_:_:_:_:.:.:.:.|_:_|
|.:.:.|_:_:_:_:_|_:_:_:_:_:.|
|_:_:_:_:.:.:.:.|_:_:_:_:_|_|
|_:_:_:_|_:_:_:_:_:.:.:.:.|_|
|.:.:.:.|_:_:_:_:_|_:_:_:_:_|
|_:_:_:_:_:_:_:_:_|_:_:_:_:_|}}\qquad\qquad\qquad\qquad\qquad
\end{center}
\caption{The strategy of sage $A$. Here $h(A)=14$, $h(B)=14$, $w_A=6$, $w_B=5$. In the construction of the table it is not required, but to complete the picture one can assume that $g(A)=4$, $g(B)=4$.}\label{fig:interaction-tableAB}%
\end{figure}

\begin{lemma}
Consider arbitrary $w_A$ consecutive rows of this table (taking into account its toric nature, i.\,e.\ one can take several lower rows and the corresponding number of upper rows). Then each column of the table contains at most $g(A)$ empty cells in these rows.
\end{lemma}

\begin{proof} In view of toriс nature of the table it is sufficient to verify this statement for the set of first $w_A$ rows. Consider $j$-th column. It is evident that this column contains letter ``$L$'' in the entry at $i$-th row ($0\leq i\leq w_A-1$) if and only if
\begin{equation}\label{eqn:l-letters}
0\leq (j-iw_B) \bmod h(B)\leq w_B-1.
\end{equation}
In the integer sequence $d_i(j)=j-iw_B$ the distance between $d_0(j)$ and $d_{w_A-1}(j)$ is equal to
$$
(w_A-1)w_B.
$$
By condition (iii) the inequality holds:
$$
(w_A-1)w_B \geq (w_A-g(A) )h(B) -w_B,
$$
which means that for each $j$ inequality \eqref{eqn:l-letters} has at least $w_A-g(A)$ solutions for variable $i$, i.\,e.\ each column of the table contains at least $w_A-g(A)$ letters~``$L$'' in the chosen $w_A$ rows. Thus, it contains at most $g(A)$ empty squares.
\end{proof}

The hint that sage $A$ receives from the adversary is actually a set of $w_A$ consecutive rows of the table. Then the strategy of sage~$A$ is to name the colors, corresponding to the numbers of rows with empty cells in the $j$-th column of the table, where $j$ is the color of $B$'s hat. Sage $A$ can do it, because by the lemma the rows in the adversary's hint contain at most $g(A)$ empty cells in $j$-th column.

Describe the strategy of sage $B$. He sees color $i$ of the hat of sage $A$ and concludes that $A$ does not guess correctly only in the cases, when $B$'s color corresponds to the columns containing letter ``$L$'' in~\hbox{$i$-th} row. Therefore, $B$ may think that his own color is given by the set of these $w_B$ columns, and, receiving this hint, he plays with this hint by the strategy for graph~$G$.

The theorem is proven.

\section{Hat guessing numbers of some classes of graphs}
\label{sec:hgn-of-classes}

\subsection{Paths}

Blažej et al. using their theory of fractional hat guessing numbers proved \cite[Proposition 18]{blazej_bears_2021} that for each $\varepsilon>0$ there exists $N$ such that for each $n>N$ we can choose functions $h$ and $g$ on $P_n$ for which $4-\varepsilon < \min_{v\in V(P_n)} h(v)/g(v) < 4$ and game  $\langle P_n, h, g\rangle$ is winning. We confirm this result here by the exact evaluation of $\hgn_s(P_n)$ for large $n$, more precisely, we will prove that 
 game $\langle P_n, \cnst{(4s-2)}, \cnst{s} \rangle$ is winning for $n\ge 2s$,
and game $\langle P_n, \cnst{(4s-1)}, \cnst{s} \rangle$ is losing for all $n$.

Let us define hatness function on $V(P_s)=\{v_1, \dots, v_s\}$ as follows
$$
h(v_i)=\begin{cases}4s-2& \text{for } 1\leq i<s, \\ 2s-1 & \text{for } i=s.  \end{cases}
$$

\begin{theorem}\label{thm-path-P_s}
The game $\langle P_s, h, \cnst s\rangle$ is winning.
\end{theorem}

\begin{proof}
For $k=1$, $2$, \dots, $s-1$ denote by $P_k$ a path on vertices $v_1$, \dots, $v_k$ (it is a subgraph of $P_{s}$). Function $h$ allows us to define the hatnesses of vertices $v_1$, \dots, $v_k$. Check by induction on $k$ ($1\leq k\leq s$) that game with hint  $\langle P_k,\, h, \,\cnst s, \,v_k, \,s+k-1\rangle$ is winning (recall that in this game the adversary pointed to the sage $v_k$ range of $s+k-1$ consecutive colors containing the color of his hat).

Base case $k=1$: in game $\langle P_1, h, \cnst s, v_1, s \rangle$ the only player $v_1$ wins due to a hint.

Inductive step $k\to k+1$, where $k\leq s-2$. Let game with hint $\langle P_k, \, h, \, \cnst s, \, v_k, \, {s+k-1}\rangle$ be winning. Then game $\langle P_{k+1}, \,h, \,\cnst s, \,v_{k+1}, \,s+k\rangle$ is winning by theorem~\ref{thm:hint_constructor} too: here 
\begin{align*}
 B&=v_k,     &  w_B&=s+k-1, & \HG&=\langle P_k, \ h, \ \cnst s, \ v_k, \ s+k-1\rangle,\\ A&=v_{k+1}, &  w_A&=s+k,   & \tilde \HG &=\langle P_{k+1}, \ h, \ \cnst s, \ v_{k+1}, \ s+k\rangle.   
\end{align*}
Condition ii) of theorem holds because $h(A)=h(B)$, and the condition iii) is provided by the inequality
$$
w_A w_B  =  (s+k-1)(s+k) \underset{(*)}{\geq} k(4s-2) = (w_A-g(A))h(B),
$$
where inequality $(*)$ is reduced to evident inequality $(s-k)^2\geq s-k$.

The last step $k=s-1 \to s$ also holds by theorem~\ref{thm:hint_constructor}. It is verified similarly with the only difference that condition ii) holds because the number $w_B = 2s-2$ is even, and therefore $w_B\cdot h(A)= (2s-2)(2s-1)$ is divisible by $h(B)=4s-2$.

Thus, we have proved that game with hint $\langle P_s,\, h,\, \cnst s,\, v_s,\, 2s-1\rangle$ is winning. But then the game $\langle P_s, h, \cnst s\rangle$ is evidently also winning.
\end{proof}

\begin{corollary}\label{cor:s-path}
The game $\HG=\langle P_{2s}, \cnst{(4s-2)}, \cnst s\rangle$ is winning.
\end{corollary}

\begin{proof}
Denote by $\HG(A)=\langle P_s, h, \cnst s\rangle$ the game from theorem \ref{thm-path-P_s}, where $A\in V(P_s)$ is the vertex for which $h(A)=2s-1$.
Then by corollary \ref{thm:multiplication} 
$$
\HG= \HG(A)\gtimes{A} \text{\pathgame($\tfrac 1{2}$,$\tfrac 1{2}$,$A$,$B$)} \gtimes{B}\HG(B)
$$ 
is winning.
\end{proof}

\begin{theorem}\label{thm: path:s-lose}
The game $\langle P_n, \cnst{(4s-1)}, \cnst{s} \rangle$ is losing for all $n$.
\end{theorem}

\begin{proof}
Let
$$
h(v_i)=\begin{cases}4s-1& \text{for } 2\leq i\leq n, \\ 2s & \text{for } i=1.  \end{cases}
$$
It is sufficient to check that game $\langle P_n, h, \cnst{s} \rangle$ is losing for all $n$. For each natural $s$, we prove this statement by induction on $n$. 

Base case $n=1$. The game \pathgame($\frac s{2s}$,$\frac s{4s-1}$,$A$,$B$) \quad  is losing by theorem~\ref{thm:clique-win}.

Induction step. Consider the two leftmost vertices $A$ and $B$ and all possible hat color assignments to sage $B$. Sage $A$ names $s(4s-1)$ colors from set $\{0,1, 2, \dots, 2s-1\}$ in total. Therefore some color $c_A$ occurs in his answers at most $[\frac{s(4s-1)}{2s}]=2s-1$ times. Give to sage~$A$ the hat of this color. Then sage $B$ sees color $c_A$ and knows, for which $2s-1$ colors of his hat sage $A$ names color~$c_A$. So, sage $B$ may assume that the color of his own hat is taken from the set $C_B$ consisting of $4s-1-(2s-1)=2s$ colors. At that moment, the adversary declares\footnote{ Formally, this escapade is a violation of game rules. The adversary should have inform the sages about possibility of this event before they start to discuss their strategies.}  that in current hats placement the color of $B$'s hat belongs to $C_B$ and inform the other sages what is the set $C_B$.
Then the game from induction step takes place on the remained graph and it is losing.
\end{proof}

Combining corollary \ref{cor:s-path} and theorem \ref{thm: path:s-lose} we obtain the following theorem.

\begin{theorem}\label{thm:s-pathclass}
$\hgn_s(\mathcal{P}\hskip-.5pt ath)= 4s-2$. More precisely, 

a) $\hgn_s(P_n) < 4s-1$ for all $n$. 

b) $\hgn_s(P_n) = 4s-2$ for $n\geq 2s$.
\end{theorem}

\subsection{Petals and royal petunias}

Recall that the \emph{star} graph $K_{1,n}$ is a tree of $n+1$ vertices, where one of the vertices, denote it by $A$, is a root and all other are leaves. We need the following lemma about star graphs which is similar to \cite[theorem 7]{butler2009hat}. For any positive integers $s$ and $H$ let  $h_{s,H}$, $g_s$ be the following functions on $V(K_{1,n})$:
$$
h_{s,H}(v)=\begin{cases}
    s+1& v\ne A,\\
    H  & v=A,
\end{cases}
\qquad
g_s(v)=\begin{cases}
    1& v\ne A,\\
    s& v=A.
\end{cases}
$$

\begin{lemma}\label{lem:star}
For any positive integers $s$ and $H$ there exists integer $n$ such that the game $\langle K_{1,n}, \ h_{s,H}, \ g_s \rangle $ is winning. 
\end{lemma}

\begin{proof}
By a ``scrap-heap'' we mean $s+1$ heaps of stones containing $H$ stones in total (the stones are numbered from 0 to $H-1$, and the heaps are numbered by 0, 1, \dots, $s$ i.\,e.\ by possible hat colors of peripheral sages, empty heaps are allowed). Let $n=H^{s+1}$ be the number of all possible scrap-heaps. Define the strategy of the sages on graph $K_{1,n}$. Give a unique scrap-heap to each sage $B_i$. The strategy of $B_i$ is to name the index of the heap containing the stone $c_A$.  The strategy of sage $A$ is to enroll those colors of his own hat, for which none of $B_i$ has guessed correctly, and to name all listed colors. This is possible because the list contains at most $s$ colors. Indeed, if the list contains colors $c_0$, $c_1$, \dots, $c_{s}$, then consider any scrap-heap, in which for all~$i$ the stone $c_{i}$ lies in $i$-th heap. Without loss of generality one can assume that the owner of the scrap-heap has received a hat of color~0 in  current hats placement. But then he certainly correctly guesses his own color, if the sage $A$ has received hat of color $c_0$ that contradicts the definition of~$c_0$.
\end{proof}

The size of the star $n$ in the proof above can be drastically decreased if we use special combinatorial tools. A~\emph{perfect hash family} $\text{PHF}(N; k, v, t)$ is an $N \times k$ array on $v$ symbols, in which in every $N \times t$ subarray, at least one row is comprised of distinct symbols. A~\emph{perfect hash family number} $\text{PHFN}(k, v, t)$ is the smallest value $N$ for which $\text{PHF}(N; k, v, t)$ exists.

\begin{corollary}
    The game $\langle K_{1,n}, \ h_{s,H}, \ g_s \rangle $ is winning if and only if \\ 
    $\text{\normalfont{PHF}}(n; H, s+1, s+1)$ exists.
\end{corollary}

\begin{proof}
    One can describe a strategy of peripheral sages by $n \times H$ array on $s+1$ symbols (=colors): the entry in column $c_A$ in the $i$-th row is the guess of the $i$-th sage (=one of $s+1$ symbols). 

    If $\text{PHF}(n; H, s+1, s+1)$ exists, its $n \times H$ array defines the strategy of peripheral sages. Let the central sage $A$ use the strategy from the lemma proof: he names the list of ``bad'' colors. The number of bad colors does not exceed $s$ due to the same contradiction: for any set of $s+1$ columns(=colors of $A$) there exists a row(=sage) that is comprised of distinct colors, but then the color $c_0$ of the corresponding sage's hat cannot be bad. Thus this strategy is winning.

    In the opposite direction: if the sages have a winning strategy, the corresponding $n \times H$ array is PHF. Otherwise, there exist $s+1$ columns of array for which each row contains at most $s$ distinct symbols. Therefore, we can assign colors to peripheral sages so that none of them will guess correctly, if $A$'s hat is one of these $s+1$ colors. After that, assign the color to $A$ such that $A$ will not guess too.
\end{proof}

\begin{corollary}\label{cor:less-order-of-star}
$1)$ The game $\langle K_{1,n}, \ h_{2,22}, \ g_2 \rangle $ is winning for $n\geq 9$. 

$2)$ For any positive integers $s$ and $H$ and $n\geq \text{\normalfont{PHFN}}(H, s+1, s+1)$  the game $\langle K_{1,n}, \ h_{s,H}, \ g_s \rangle $ is winning.
\end{corollary}

\begin{proof} $\text{PHF}(9; 27, 3, 3)$ exists due to result of van Trung and Martirosyan \cite[Remark 4.1]{trung2005new}, this immediately gives us the first statement. The second statement follows from the previous corollary.
\end{proof}

\begin{theorem}\label{thm:s-petal}
$\hgn_s(\mathcal{P}\hskip-.5pt etal)= 4s(s+1)-2$. More precisely, 

a) $\hgn_s(\petal_n) < 4s(s+1)-1$ for all $n$. 

b) $\hgn_s(\petal_n) = 4s(s+1)-2$ for large $n$.
\end{theorem}

\begin{proof} 
a)  Consider an arbitrary petal. Let $A$ be its stem, $h(A)=s+1$, and let the other vertices $v$ have hatness $h(v)=4s(s+1)-1$. It is sufficient to verify that the game $\langle G, h, \cnst{s} \rangle$ is losing. By lemma \ref{lem:(s+1 s)vertex-removing}, this game is equivalent to the game 
$ \bigl\langle P_n, \ \cnst{4s(s+1)-1}, \ \cnst{s(s+1)}\bigr\rangle $, that is losing by theorem~\ref{thm: path:s-lose}.

b) Let $\HG_0=\bigl\langle P_k, \ \cnst{4s(s+1)-2}, \ \cnst{s(s+1)} \bigr\rangle$. For $k\geq 2s(s+1)$ this game is winning by corollary~\ref{cor:s-path}. Set $H=4s(s+1)-2$ and choose 
$$
n\geq \text{PHFN}(H, s+1, s+1),
$$ 
then the game $\langle K_{1,n}, \ h_{s,H}, \ g_s \rangle $ is winning by corollary \ref{cor:less-order-of-star}. Then we apply corollary \ref{cor:to-construct-petunia} to substitute with reducing game~$\HG_0$ in the place of each of $n$ peripheral sages. We obtain a winning game, where the hatness and the guessing number of all vertices are equal to $4s(s+1)-2$ and~$s$, and the graph is a subgraph of a large petal.
\end{proof}

\begin{corollary}
$\hgn_s(\mathcal{RP})= 4s(s+1)-2$. More precisely, 

a) $\hgn_s(G) < 4s(s+1)-1$ for all royal petunias $G$. 

b) $\hgn_s(G) = 4s(s+1)-2$ for royal petunias $G$ that contain a large petal.
\end{corollary}

\begin{proof}
a) As we have checked in theorem \ref{thm:s-petal}~a), game $\langle G, h, \cnst{s} \rangle$ is losing, where $G$ is a petal 
and $h$ is a hatness function such that the hatness of the stem equals $s+1$, and the hatnesses of the other vertices are $4s(s+1)-1$. By theorem \ref{thm:s-lose-sum} gluing of stem of such petal to a vertex of another losing game with $s$ guesses gives again a losing game. But royal petunia by definition is constructed by consecutive stem's gluings of petals! Therefore a game on a royal petunia,
where all the vertices have hatness $4s(s+1)-1$ and $s$ guesses, except the first (rooted) stem with hatness $s+1$ and $s$ guesses, is losing. It gives the estimation $\hgn_s(G)<4s(s+1)-1$.

b) Then $\hgn_s(G)=4s(s+1)-2$, since this hatness is already achieved on large petals that are royal petunias too, although they are not very branchy.
\end{proof}

Since royal petunias are outerplanar graphs, we have the following corollary.

\begin{corollary}
$\hgn_s(\mathcal{OP}) \geq  4s(s+1)-2$. 
\end{corollary}

\subsection{Lower estimation for  planar graphs}

In this section, we provide an example of planar graph $G$ with $\hgn(G)\geq 22$. 

\begin{theorem} 
$\hgn(\mathcal{P}\hskip-.5pt lanar) \geq  22$. 
\end{theorem}

\begin{proof}
   Let $\petal_n$ be a petal, $\HG =\langle \petal_n, \cnst{22}, \cnst 2\rangle$ (see fig.~\ref{fig:petal-plus}, left), $\widetilde\HG=\langle \widetilde G, \cnst{22}\rangle$ be the game depicted in fig.~\ref{fig:petal-plus}, right, and let $v$ be lowest vertex of graph $\widetilde G$.

   Game $\HG =\langle \petal_n, \cnst{22}, \cnst 2\rangle$ for sufficiently large $n$ is winning by theorem \ref{thm:s-petal}, and then game $\widetilde\HG$ is winning by example~\ref{ex:1-22-planar}.
   Taking by corollary \ref{thm:multiplication} the product of $k\geq 5$ copies of game $\widetilde\HG$ with respect to vertex $v$ we obtain a winning game on planar graph with hatness function that equals 22 for all vertices except $v$ and $h(v)= 2^k\geq 32$ is as large as we wish. 
\end{proof}

The similar construction of a planar graph for an arbitrary $s$ provides the estimation $\hgn_s(\widetilde G)\geq 4(s+1)(s+2)-2$.

\begin{remark}
The minimal size of graph presented in the proof of the previous theorem is 546. Indeed, for $s=2$ game $\HG_0$ in the proof of theorem \ref{thm:s-petal} is a game $\langle P_k, \cnst{22}, \cnst{6} \rangle$. The minimum $k$ for which we know it is winning is $k=12$. We substitute this game on the leaves of graph $K_{1,n}$, for which game $\langle K_{1,n}, \ h_{s,H}, \ g_s \rangle $ must be winning.  
By corollary \ref{cor:less-order-of-star}, it happens for $n\geq 9$ and we obtain the graph on $12\cdot 9+1=109$ vertices. Then we add one more vertex as in fig.~\ref{fig:petal-plus} and after that multiply 5 copies of the graph with respect to some common vertex. The final graph has $109\cdot5+1=546$ vertices.  
\end{remark}

\bibliographystyle{elsarticle-harv}
\bibliography{main}

\end{document}